\documentclass[12pt,a4paper]{amsart}
\usepackage{amsfonts,amscd,amssymb,amsthm}
\usepackage{amsfonts, comment}
\usepackage{amssymb}
\usepackage{amsmath}
\usepackage{amsxtra}
\usepackage{amscd}
\usepackage{amsthm}
\usepackage{eucal}
\usepackage{array}
\usepackage{graphicx}
\usepackage{tikz}
\usepackage{mathrsfs} 
\usepackage{enumerate}
\usepackage{hyperref}
\usepackage{bm} 

\newtheorem{theorem}{Theorem}[section]
\newtheorem{lemma}[theorem]{Lemma}

\newtheorem{corollary}[theorem]{Corollary}

\theoremstyle{remark}
\newtheorem{remark}[theorem]{Remark}

\theoremstyle{definition}

\newtheorem{definition}[theorem]{Definition}

\DeclareMathOperator{\lcm}{lcm}
\DeclareMathOperator{\reg}{reg}
\DeclareMathOperator{\supp}{supp}

\begin{document}
%
\title{Density of linearity index in the interval of matching numbers}

\author[N. Erey]{Nursel Erey}
\author[T. Hibi]{Takayuki Hibi}
\address{Nursel Erey, Gebze Technical University, Department of Mathematics, 41400 Gebze, Kocaeli, Turkey}
\email{nurselerey@gtu.edu.tr}

\address{Takayuki Hibi, Department of Pure and Applied Mathematics, Graduate School of Information Science and Technology, Osaka University, Suita, Osaka 565--0871, Japan}
\email{hibi@math.sci.osaka-u.ac.jp}

\dedicatory{Dedicated to the memory of Jürgen Herzog}



\begin{abstract}
Given integers $2 \leq p \leq c \leq q$, we construct a finite simple graph $G$ with $\nu_1(G) = p$ and $\nu(G) = q$ for which the squarefree power $I(G)^{[k]}$ of the edge ideal $I(G)$ of $G$ has linear quotients for each $c \leq k \leq q$ and is not linearly related for each $1 \leq k < c$, where $\nu_1(G)$ is the induced matching number of $G$ and $\nu(G)$ is the matching number of $G$.
\end{abstract}

\maketitle

\section{Introduction}
Let $G$ be a finite simple graph on the vertex set $[n] = \{1, \ldots, n\}$ and let $E(G)$ be its edge set.  Recall that a finite graph $G$ is {\em simple} if $G$ has no loops and no multiple edges.  Let $S = K[x_1, \ldots, x_n]$ denote the polynomial ring in $n$ variables over a field $K$.  The {\em edge ideal} of $G$ is the monomial ideal $I(G) = (x_ix_j : \{i,j\} \in E(G)) \subset S$.  Fr\"oberg \cite{Frö} showed that $I(G)$ has linear resolution if and only if the complementary graph $\overline{G}$ of $G$ is chordal, where $\overline{G}$ is a finite simple graph on $[n]$ with $E(\overline{G}) = \{\{i,j\} : i < j, \{i,j\} \notin E(G)\}$.  (A chordal graph is a finite simple graph whose cycles of length $ > 3$ have a chord.)  Later, in 2004, by virtue of Dirac’s theorem on chordal graphs together with the modern theory of Gr\"obner bases, it was shown in \cite{HHZ} that all powers of $I(G)$ have linear resolution if and only if $\overline{G}$ is chordal.  Under the influence of \cite{HHZ} a huge number of papers studying powers of edge ideals have been published.  

The squarefree power of an edge ideal was first studied in \cite{BHZ}.  A set of edges $M = \{e_1, \ldots, e_k\}$ of $G$ is called a {\em matching} of $G$ if $e_i \cap e_j = \emptyset$ for $1 \leq i < j \leq k$.  If a matching $M$ consists of $k$ edges, then $M$ is called a \emph{$k$-matching}.  The {\em matching number} $\nu(G)$ of $G$ is the maximal cardinality of the matchings of $G$. For an edge $e=\{i,j\}$ let $x_e$ denote the monomial $x_ix_j$. Given a matching $M = \{e_1, \ldots, e_k\}$, the notation $x_M$ stands for the squarefree monomial $x_{e_1}\dots x_{e_k}$ of $S$.  The $k$th {\em squarefree power} of the edge ideal $I(G)$ of $G$ is defined to be the squarefree monomial ideal 
\[
I(G)^{[k]} = (\{ x_M : \text{$M$ is a $k$-matching of $G$} \})
\]   
of $S$, where $1 \leq k \leq \nu(G)$.  Thus in particular $I(G) = I(G)^{[1]}$. Note that $I(G)^{[k]}$ is generated by the squarefree monomials in $I(G)^k$. In \cite{BHZ} it is proved that $\nu(G)$th squarefree  power $I(G)^{[\nu(G)]}$ has linear quotients. Following \cite{BHZ}, squarefree powers of edge ideals were studied in detail in \cite{EHHS}. One can refer to  \cite{CF, CFL, DRS1, DRS2, EF1, EF2, EHHS2, F1, F2, F3, FHH, FM, HS, KNQ} for some recent research on this topic.

An {\em induced matching} of $G$ is a matching $M = \{e_1, \ldots, e_k\}$ for which the induced subgraph on $\cup_{i=1}^{k} e_k$ consists of exactly $k$ edges.  The {\em induced  matching number} $\nu_1(G)$ of $G$ is the maximal cardinality of the induced matchings of $G$. The \emph{minimum matching number} of $G$, denoted by $\nu_2(G)$, is the minimum cardinality of the maximal matchings of $G$. It is well-known from \cite{HVT, K, W} that $\nu_1(G)\leq \reg(S/I(G))\leq \nu_2(G)\leq \nu(G)$. Hibi et al. \cite{HHKT} studied non-decreasing sequences $(a_1, a_2, a_3, a_4)$ for which there exists a graph $G$ with $(a_1, a_2, a_3, a_4)=(\nu_1(G),\, \reg(S/I(G)),\, \nu_2(G),\, \nu(G))$. 

In this article, we consider a question of similar type for squarefree powers. Let $c(G)$ denote the smallest integer $c$ for which $I(G)^{[c]}$ has linear resolution.  We say that $c(G)$ is the {\em linearity index} of $G$.  It follows from \cite[Theorem ~2.1]{EHHS} and \cite[Theorem~4.1]{BHZ} that linearity index is bounded between the induced matching number and the matching number of $G$: \[\nu_1(G)\leq c(G) \leq \nu(G).\]
It is natural to ask if, given integers $1 \leq p \leq c \leq q$, there is a finite simple graph $G$ with 
\[\nu_1(G) = p,\ c(G) = c \ \text{ and } \ \nu(G) = q.\]  The main result of the present paper is Theorem~\ref{maintheorem}, which guarantees that, given integers $2 \leq p \leq c \leq q$, there exists a finite simple connected graph $G$ with $\nu_1(G) = p, \nu(G) = q$ and $c(G) = c$ for which $I(G)^{[k]}$ has linear quotients for each $c \leq k \leq q$ and $I(G)^{[k]}$ is not linearly related for each $1 \leq k < c$. We remark that the case when $\nu_1(G)=p = 1$ is not included in Theorem~\ref{maintheorem}. It is therefore unclear if the linearity index of a so-called \textit{gap-free} graph could be any integer in the interval $[\nu_1(G),\, \nu(G)]$.

It was proved by the authors of this paper in \cite{EH} that if $G$ is a forest, then $I(G)^{[k]}$ has linear resolution for all $k\geq c(G)$. The family of graphs (Definition \ref{def}) we construct in the main theorem of this paper also has this property. It is unknown though whether such a property holds for \textit{any} graph $G$.


\section{Main results}
 Let $I$ be a monomial ideal in $S$ generated in degree $d$. We denote by $G(I)$ the set of minimal monomial generators of $I$. We define the graph $G_I$
 whose vertex set is $G(I)$ and its edge set is
 \[E(G_I) = \{\{u,v\} : u,v \in G(I) \text{ with } \deg(\lcm(u,v)) = d+1\}.\]
 For all $u,v \in G(I)$ let $G^{(u,v)}_I$ 
 be the induced subgraph of $G_I$ with vertex set
 \[V(G^{(u,v)}_I) = \{w \in G(I): w \text{ divides } \lcm (u,v)\}.\]
Let $I \subset S$ be a graded ideal generated in a single degree. The ideal $I$ is said to
 be \emph{linearly related}, if the first syzygy module of $I$ is generated by linear relations. It was shown in \cite{BHZ} that $I$ being linearly related is determined by the graph $G_I$.
 \begin{theorem}\cite[Corollary~2.2]{BHZ}
Let $I$ be a monomial ideal generated in degree $d$. Then $I$ is linearly
 related if and only if for all $u,v \in G(I)$ there is a path in $G^{(u,v)}_I$
 connecting $u$ and $v$.  
 \end{theorem}

 We denote a \textit{complete} graph on $m$ vertices by $K_m$. We call a vertex $v$ of $G$ a \textit{whisker} if $v$ is adjacent to exactly one vertex of $G$. An \textit{isolated vertex} of $G$ is a vertex that is not adjacent to any vertex of $G$. To simplify the notation, we will identify the vertices of a graph with the variables of a polynomial ring throughout this paper.
 \begin{lemma}
 \label{related}
Let $1\leq d<c$ be integers. Let $H$ be a graph on $c$ vertices whose connected components are a complete graph $K_{c-d}$ and $d$ isolated vertices. Let $G$ be the graph which is obtained by adding a whisker to each vertex of $H$. Then $I(G)^{[c-1]}$ is not linearly related and thus does not have linear resolution.
 \end{lemma}
 \begin{proof}
     Let $V(H)=\{x_1,\dots ,x_c\}$ where $x_1,\dots ,x_d$ are isolated vertices and the vertices $x_{d+1},\dots ,x_c$ form a clique. Let $G$ be obtained by attaching the whisker $y_i$ to $x_i$ for every $i=1,\dots ,c$. Let $w=x_1\dots x_cy_1\dots y_c$. Let $J=I(G)^{[c-1]}$. Consider $u=w/(x_1y_1), \, v=w/(x_{d+1}y_{d+1})\in G(J)$. Then $w=\lcm(u,v)$. We claim that there is no path in $G_J^{(u,v)}$ connecting $u$ to $v$. Now, observe that $G(J)=V(G_J^{(u,v)})$ so that $G_J=G_J^{(u,v)}$. Then we can write $V(G_J^{(u,v)})=V_1\cup V_2$ for 
     \[V_1=\Big\{\frac{w}{x_iy_i}: i\in \{1,\dots ,c\}\Big\}\]
     and 
     \[V_2=\Big\{\frac{w}{(x_iy_i)(x_jy_j)}(x_ix_j): i\neq j \text{ and } i,j\in \{d+1,\dots ,c\} \Big\}. \]
     By definition of the graph $G_J$, no two vertices in $V_1$ are adjacent. Moreover, no vertex in $V_1$ is adjacent to a vertex in $V_2$. Since both $u$ and $v$ are in $V_1$, it follows that they are isolated vertices of the graph $G_J$.
     \end{proof}

     \begin{corollary}
    Let $G$ be the same graph as in Lemma \ref{related}.  Then $I(G)^{[k]}$ is not linearly related for $k < c$.
     \end{corollary}

     \begin{proof}
    We proceed by induction on $d > 0$. First, we discuss the case of $d = 1$.  To prove the corollary for $d=1$, we use induction on $c > 1$.  If $c=2$, then $G$ is the disjoint union of two edges and thus $I(G)$ is not linearly related.  Let $c > 2$ and let $G'$ be the induced subgraph obtained by removing one of the vertices of $K_{c-1}$ and its whisker from $G$. Then the assumption of induction says that $I(G')^{[k]}$ is not linearly related for $k < c-1$.  Since $G'$ is an induced subgraph, it follows from \cite[Corollary 1.3]{EHHS} that $I(G)^{[k]}$ is not linearly related for $k < c-1$.  Then, the desired result follows from Lemma \ref{related}.
    
    We now assume that $d > 1$. Let $G'$ be the induced subgraph obtained by removing one of the isolated vertices of $H$ and its whisker from $G$. Then the assumption of induction says that $I(G')^{[k]}$ is not linearly related for $k < c-1$.  Since $G'$ is an induced subgraph, it follows again that $I(G)^{[k]}$ is not linearly related for $k < c-1$. Then the desired result follows from Lemma \ref{related}.
    \end{proof}

Let $u$ be a minimal generator of $I(G)^{[s]}$. A monomial $x_ix_j$ corresponding to an edge $\{i,j\}$ of $G$ is called an \textit{$s$-fold component} of $u$ if $u/(x_ix_j)$ is a minimal generator of $I(G)^{[s-1]}$. Given two monomials $u$ and $v$, we will denote by $u:v$ the monomial $\frac{u}{\gcd(u,v)}$. By $\supp(u)$, we will denote the set of variables dividing the monomial $u$.

Let $I \subset S$ be a monomial ideal generated in one degree and let $G(I)=\{u_1, \ldots, u_s\}$. We say that $I$ has {\em linear quotients} if there exists an ordering $u_{i_1}, \ldots, u_{i_s}$ of $u_1, \ldots, u_s$ for which the colon ideal $(u_{i_1}, \ldots, u_{i_{j-1}}):u_{i_j}$ is generated by variables for each $2 \leq j \leq s$.

Let $G$ and $H$ be two graphs with disjoint vertices. The \textit{join} of $G$ and $H$, denoted by $G*H$, is the new graph whose vertex set is $V(G)\cup V(H)$ and edge set is \[E(G*H)=E(G)\cup E(H)\cup\{\{x,y\}:x\in V(G), y\in V(H)\}.\]

The \textit{reverse lexicographic order} induced by $x_{i_1}>x_{i_2}>\dots >x_{i_n}$ is the total order on the monomials of $S=K[x_1,\dots ,x_n]$ such that $x_{i_1}^{\alpha_1}\dots x_{i_n}^{\alpha_n} > x_{i_1}^{\beta_1}\dots x_{i_n}^{\beta_n}$ whenever the rightmost non-zero component of $(\beta_1-\alpha_1, \beta_2-\alpha_2, \dots , \beta_n-\alpha_n)$ is positive.

\begin{theorem}
\label{nursel}
Let $1\leq r \leq q$ be integers. Let $H_1$ be any graph on $r$ vertices and let $H_2$ be a complete graph on $q-r$ vertices such that $V(H_1)$ and $V(H_2)$ are disjoint. Let $G$ be the graph which is obtained by adding a whisker to each vertex of $H_1*H_2$. Then $I(G)^{[k]}$ has linear quotients for all $r\leq k\leq q$. 
\end{theorem}
\begin{proof}
    Let $V(H_1)=\{a_1,\dots ,a_r\}$ and $V(H_2)=\{c_1,\dots ,c_{q-r}\}$. Let $b_i$ be the whisker attached to $a_i$ for all $1\leq i\leq r$. Let $d_i$ be the whisker attached to $c_i$ for all $1\leq i\leq q-r$. We consider the reverse lexicographic order induced by the following order of the variables: \[a_1>\dots >a_r>b_1>\dots >b_r>c_1>\dots >c_{q-r}>d_1>\dots >d_{q-r}.\] 
    
    Let $r\leq k\leq q$ and let $u_1>\dots >u_p$ be the minimal generators of $I(G)^{[k]}$ in the reverse lexicographic order. We claim that $u_1,\dots ,u_p$ is a linear quotients order. Let $u_j>u_i$. We will show that there exists $u_t>u_i$ such that $u_t:u_i$ is a variable dividing $u_j:u_i$.

 \textbf{Case 1:} Suppose that there exists some $a_\ell$ which is in $\supp(u_j)\setminus \supp(u_i)$. Since $k\geq r$ the monomial $u_i$ has a $k$-fold component of the form $w=c_yd_y$ or $w=c_yc_z$. Otherwise, the support of $u_i$ contains $\{a_1,\dots,a_r\}$ against the fact that $a_\ell\in \supp(u_j)\setminus \supp(u_i)$. In either case, $\frac{u_i}{w}c_ya_\ell=u_t$ for some minimal generator $u_t$. Then $u_t>u_i$ and $u_t:u_i=a_\ell$.

 \textbf{Case 2:} Suppose that $\supp(u_j)\cap V(H_1)\subseteq \supp(u_i)\cap V(H_1) $. 
	
	\textbf{Case 2.1:} Suppose that there exists $c_\ell$ in $\supp(u_j)\setminus \supp(u_i)$. If there exists $d_y$ dividing $u_i$, then $\frac{u_i}{d_y}c_\ell=u_t$ for some minimal generator $u_t$ which satisfies $u_t>u_i$ and $u_t:u_i=c_\ell$. Suppose that there is no $d_y$ dividing $u_i$. Then since the generators are ordered in the reverse lexicographic order, no $d_y$ divides $u_j$ either. Moreover, since $u_j>u_i$, there exists $\ell'>\ell$ such that $c_{\ell'}$ divides $u_i$. Then $\frac{u_i}{c_{\ell'}}c_\ell=u_t$ for some minimal generator $u_t$ which satisfies $u_t>u_i$ and $u_t:u_i=c_\ell$.

 \textbf{Case 2.2:} Suppose that $\supp(u_j)\cap V(H_2)\subseteq \supp(u_i)\cap V(H_2)$. Then since $u_j\neq u_i$ the monomial $u_j:u_i$ is divisible by $d_\ell$ or $b_\ell$ for some $\ell$.
 
 \textbf{Case 2.2.1:}
First, suppose that $u_j:u_i$ is divisible by some $d_\ell$ so that $c_\ell d_\ell$ is a $k$-fold component of $u_j$. Since $u_j>u_i$ with respect to reverse lexicographic order, it follows that $c_{\gamma}d_{\gamma}$ is a $k$-fold component of $u_i$ for some $\gamma>\ell$. Since $c_\ell$ divides $u_i$, the monomial $\frac{u_i}{c_{\gamma}d_{\gamma}}$ has a $(k-1)$-fold component of the form $c_\ell a_{\ell'}$ or $c_\ell c_{\ell'}$. If $c_\ell a_{\ell'}$ is a $(k-1)$-fold component, then $w=\frac{u_i}{(c_\ell a_{\ell'})(c_{\gamma}d_{\gamma})}$ is a minimal generator of $I(G)^{[k-2]}$. Then $w(c_\ell d_\ell)(a_{\ell'} c_{\gamma})=u_t$ for some $u_t>u_i$ and $u_t:u_i=d_\ell$.
	
	On the other hand, if $c_\ell c_{\ell'}$ is a $(k-1)$-fold component,  then $w=\frac{u_i}{(c_\ell c_{\ell'})(c_{\gamma}d_{\gamma})}$ is a minimal generator of $I(G)^{[k-2]}$. Then $w(c_\ell d_\ell)(c_{\gamma}c_{\ell'})=u_t$ for some $u_t>u_i$ and $u_t:u_i=d_\ell$ as desired.
 
 \textbf{Case 2.2.2:} We may now assume that $\supp(u_j)\cap D\subseteq \supp(u_i)\cap D$ where $D=\{d_1,\dots ,d_{q-r}\}$. Suppose that $u_j:u_i$ is divisible by some $b_\ell$ so that $a_\ell b_\ell$ is a $k$-fold component of $u_j$. Then either $a_\ell c_{\ell'}$ or $a_\ell a_{\ell'}$  is a $k$-fold component of $u_i$ for some $\ell'$. First, suppose that $a_\ell c_{\ell'}$ is a $k$-fold component of $u_i$. Then $\frac{u_i}{c_{\ell'}}b_\ell=u_t$ for some minimal generator $u_t$ which satisfies $u_t>u_i$ and $u_t:u_i=b_\ell$. Next, suppose that $a_\ell a_{\ell'}$  is a $k$-fold component of $u_i$ for some $\ell'$. Since $|V(H_1)|=r\leq k$ it follows that $u_i/(a_\ell a_{\ell'})$ has a $(k-1)$-fold component of the form $c_xz$ where $z\in V(H_2)\cup D$. Then $w=\frac{u_i}{(a_\ell a_{\ell'})(c_xz)}\in I(G)^{[k-2]}$ and $u_t:=u_ib_\ell /z=(a_\ell b_\ell)(a_{\ell '}c_x)w$
 satisfies the required condition.
 \end{proof}

 \begin{definition} 
 \label{def}
  Let $1\leq p \leq c \leq q$ be integers. Let $K_{q-c}$ be the complete graph on $q-c$ vertices. Let $H$ be the graph on $c$ vertices whose connected components are a complete graph of order $c-(p-1)$ and $p-1$ isolated vertices. Let $G_{p,c,q}$ be the graph which is obtained by adding a whisker to each vertex of $K_{q-c}*H$.
\end{definition}
One has $\nu_1(G_{p,c,q}) = p$ and $\nu(G_{p,c,q}) = q$.  The graph $G_{p,c,q}$ is exactly what we are looking for and the main result of the present paper is as follows:

\begin{theorem}
\label{maintheorem}
    Let $2\leq p\leq c \leq q$ be integers. Then $I(G_{p,c,q})^{[k]}$ has linear quotients for $k\geq c$ and is not linearly related for $k < c$.
\end{theorem}

\begin{proof}  
    Let $H$ be the graph as in Definition \ref{def} and $G$ be the graph obtained by adding a whisker to each vertex of $H$.  Then Corollary~\ref{related} says that $I(G)^{[k]}$ is not linearly related for $k < c$.  Since $G$ is an induced subgraph of $G_{p,c,q}$, it follows from \cite[Corollary~1.3]{EHHS} that $I(G_{p,c,q})^{[k]}$ is not linearly related for $k< c$.  On the other hand, Theorem \ref{nursel} gurantees that $I(G_{p,c,q})^{[k]}$ has linear quotients for $k\geq c$, as desired.
\end{proof}

\begin{remark}
When $p =1$, the graph $G_{p,c,q}$ is exactly the complete graph $K_{q}$ with a whisker at each vertex. This provides an example for any integer $q$ with $p=c=1\leq q$ because $c(G_{p,c,q}) = 1$.  Also, one can construct a graph of desired type for any integer $q$ with $p=1<c=2\leq q$ by multiplying some vertices of a $5$-cycle as described in \cite{E}. It is unclear if Theorem~\ref{maintheorem} is true for any integers $c$ and $q$ with $p =1<2<c\leq q$.
\end{remark}


\section*{Acknowledgment}
The research on this paper began while the authors were visiting Jürgen Herzog in Essen in August 2023.  We thank Jürgen Herzog together with Universität Duisburg -- Essen for the hospitality during our stay in Essen. The second author was supported by TÜBİTAK (2221 - Fellowships for Visiting Scientists and Scientists on Sabbatical Leave) to visit Nursel Erey at Gebze Technical University between 1-15 September, 2024. Macaulay 2 Package \cite{Fic} was helpful in the preparation of this paper.

\end{document}